\documentclass{amsart}
\usepackage{amsmath}
\usepackage{amssymb}
\usepackage[utf8]{inputenc}
\usepackage{enumerate} 

\newtheorem{thm}{Theorem}[section]

\newtheorem{lem}[thm]{Lemma}
\newtheorem{cor}[thm]{Corollary}

\newtheorem{gronwall}[thm]{Grönwall's Inequality}

\theoremstyle{definition}

\theoremstyle{remark}

\newtheorem{remark}[thm]{Remark}

\numberwithin{equation}{section}

\newcommand{\R}{\mathbb{R}}  
\newcommand{\N}{\mathbb{N}}
\newcommand{\Lcurv}{\mathcal{L}}

\newcommand{\norm}[1]{\left\lVert#1\right\rVert}

\newenvironment{indentpar}[1]%
 {\begin{list}{}%
         {\setlength{\leftmargin}{#1}}%
         \item[]%
 }
 {\end{list}}

\begin{document}

\title[Scalar Curvature Rigidity, Regularity, $L^p$-Bounds, Decay Rates]{Scalar Curvature Rigidity and Ricci DeTurck Flow on Perturbations of Euclidean Space}
\author{Alexander Appleton}
\address{Department of Mathematics, UC Berkeley, 
CA 94720, USA}
\email{aja44@berkeley.edu}
\maketitle

\begin{abstract}
We prove a rigidity result for non-negative scalar curvature perturbations of the Euclidean metric on $\R^n$, which may be regarded as a weak version of the rigidity statement of the positive mass theorem. We prove our result by analyzing long time solutions of Ricci DeTurck flow. As a byproduct in doing so, we extend known $L^p$ bounds and decay rates for Ricci DeTurck flow and prove regularity of the flow at the initial data.
\end{abstract}

\tableofcontents

\section{Introduction}
In this paper we will prove the following rigidity theorem for non-negative scalar curvature metrics on $\R^n$:
\begin{thm}
\label{scalar-rigidity-thm}
There exists an $\epsilon(n)>0$ such that any smooth metric $g$ on $\R^n$, $n\geq2$ which satisfies
\begin{enumerate} 
\item $R_g \geq 0 $
\item $\norm{g-g_{eucl}}_{L^{\infty}(\R^n)} < \epsilon$
\item $\norm{g-g_{eucl}}_{L^p(\R^n)} < \infty$ for some $1 \leq p < \frac{n}{n-2}$
\end{enumerate}
is isometric to the standard Euclidean metric $g_{eucl}$ on $\R^n$.
\end{thm}
It is unclear what happens in the borderline case $p=\frac{n}{n-2}$, for which we only have a partial result proven in lemma \ref{borderline}.

Theorem \ref{scalar-rigidity-thm} is related to the rigidity statement of the positive mass theorem proven by Schoen and Yau (\cite{SY}, \cite{SY2}) using minimal surface techniques. They showed that the ADM-mass 
\begin{equation}
m(g) = \lim_{r\rightarrow\infty} \int_{S_r} \left(\partial_j g_{ij} - \partial_i g_{jj}\right) \, \mathrm{d}A^i
\end{equation}
of an asymptotically Euclidean non-negative scalar curvature manifold (see \cite{ADM} and \cite{Bartnik}) of dimension $n\leq7$ is non-negative and that the manifold is flat if $m(g) =0$. The dimension $n = 8$ case can be handled by the perturbation result of Smale \cite{Sma93}. Recently, Schoen and Yau \cite{SY17} have submitted a proof generalizing the minimizing hypersurface technique to all dimensions.

These results are consistent with theorem \ref{scalar-rigidity-thm}, since for metrics satisfying 
\begin{equation}
\norm{g-g_{eucl}}_{L^p(\R^n)}<\infty
\end{equation}
we heuristically have
\begin{align*}
|g-g_{eucl}| &\sim  o\left(\frac{1}{r^{\frac{n}{p}}}\right)\\
\partial g &\sim o\left(\frac{1}{r^{\frac{n}{p} -1 }}\right)
\end{align*}
and 
\begin{equation*}
\int_{S_r} \partial_j g_{ij} - \partial_i g_{jj} \, \mathrm{d}A^i \sim o\left(r^{n-2 - \frac{n}{p}}\right).
\end{equation*}
Hence for $p < \frac{n}{n-2}$ we expect $m(g)=0$ and thus $g$ to be flat by the positive mass theorem.

We will prove theorem \ref{scalar-rigidity-thm} by evolving the metric via the Ricci DeTurck flow, which is a geometric flow related to the Ricci flow by a time-dependent family of diffeomorphisms. Its equation can be written as
\begin{equation}
\label{RdT-schematic1}
\partial_t g_t = -2Ric(g_t) - \Lcurv_{X_t} (g_t), \qquad t>0,
\end{equation} 
where $\Lcurv_{X_t}(g_t)$ is the Lie derivative of the metric $g_t$ with respect to a time varying vector field $X_t$. As we will be studying the evolution of perturbations of the Euclidean metric, it is useful to consider the quantity $h_t = g_t - g_{eucl}$. The Ricci DeTurck equation with respect to the fixed Euclidean background metric $g_{eucl}$ on $\R^n$ then takes the form 
\begin{equation}
\label{RdT-schematic2}
(\partial_t - \Delta) h_t = Q(h_t, \nabla h_t , \nabla^2 h_t)
\end{equation}
where 
\begin{align}
Q(h_t, \nabla h_t , \nabla^2 h_t) &= (g_{eucl} + h_t)^{-1}\ast(g_{eucl} + h_t)^{-1}\ast\nabla h_t \ast \nabla h_t  \\
 &\qquad\qquad\qquad\qquad+ \nabla\ast \left( \left((g_{eucl}+h_t)^{-1}- g_{eucl}^{-1}\right) \ast \nabla h\right)
\end{align}

Equation (\ref{RdT-schematic2}) is strongly parabolic, however due to its non-linearity and the non-compactness of the domain, a short time solution does not exist a priori and if it does, we expect singularities to develop after finite time precluding long time solutions. Surprisingly, it turns out that for $L^{\infty}$-small initial data a long time solution to the flow exists. 

Schnürer, Schulze and Simon \cite{MilesSimon} were the first to show long time existence to the Ricci DeTurck equation (\ref{RdT-schematic2}) for $C^0$ perturbations of the Euclidean metric, which satisfy certain $L^{\infty}$-decay conditions at infinity. They also proved that perturbations bounded in $L^p$ for $p\geq 2$ remain so at later times. Using these results and an interpolation inequality, they obtain $L^{\infty}$-decay rates of $h_t$. We will extend their results to $1\leq p < 2$ and slightly improve their decay rates using a limit argument (see our theorem \ref{main-thm}, (ii) \& (iii)). 

Koch and Lamm in \cite{KochAndLamm} were able to remove the decay conditions at infinity by constructing a weak solution and proving analyticity of the solution on $\R^n\times(0,\infty)$. We will make extensive use of the results of Koch and Lamm and extend them, by showing that for smooth initial data the weak solution they constructed is classical, e.g. the solution and its derivatives are continuous up to $t=0$. In particular we prove the following theorem:

\begin{thm}
\label{main-thm}
There exists an $\epsilon=\epsilon(n)>0$ and $C=C(n)>0$ such that for smooth initial data $h_0$ satisfying $\norm{h_0}_{L^{\infty}(\R^n)}< \epsilon$ we have 
\begin{enumerate}[(i)]
\item A smooth classical solution $h_t \in C^{\infty}(\R^n \times [0,\infty))$ to the Ricci DeTurck equation (\ref{RdT-schematic2}) exists. 
\item If $\norm{h_0}_{L^p(\R^n)}<\infty$ for some $1\leq p \leq \infty$, then $\norm{h_t}_{L^{p}(\R^n)} \leq C \norm{h_0}_{L^{p}(\R^n)}$ for all $t\geq0$. 
\item If $\norm{h_0}_{L^p(\R^n)}<\infty$ for some $1\leq p < \infty$, then the following decay rate holds for $t>0$
\begin{equation}
\sup_{x\in \R^n}|h_t(x)| \leq \frac{C}{t^{\frac{n}{2p}}}\norm{h_0}_{L^p(\R^n)}.
\end{equation}
For higher derivatives we have that there exists a $R=R(n)>0$ such that for all $k \in \N_0$ and multi-indices $\alpha \in \N^n_{0}$ we have
\begin{equation}
\label{decay-rates}
\sup_{x\in \R^n} |(t^{\frac{1}{2}}\nabla)^{\alpha}(t\partial_t)^k h_t(x)| \leq \frac{C R^{|\alpha|+k}(|\alpha| + k)!}{t^{\frac{n}{2p}}} \norm{h_0}_{L^p(\R^n)}. 
\end{equation}
\end{enumerate}
\end{thm}

The setup of our paper is as follows: We begin by proving that for sufficiently regular initial data a classical solution to the Ricci DeTurck equation exists. Then we show that initial data bounded in $L^p(\R^n)$, for some $1\leq p < \infty$, remains uniformly bounded in $L^p(\R^n)$ at later times. This will allow us to prove $L^{\infty}$ decay rates of the evolving metric and its derivatives, which are as one would expect by comparison with the standard heat equation on $\R^n$.

Finally, we use the decay rates (\ref{decay-rates}) on the metric and its derivatives to prove the scalar rigidity theorem \ref{scalar-rigidity-thm} above. The scalar curvature evolving under Ricci DeTurck flow satisfies the following super heat equation
\begin{equation}
(\partial_t - \Delta + \Lcurv_{X_t}) R(g_t) = 2 |Ric(g_t)|^2.
\end{equation}
As it is the case for bounded solutions to the standard heat equation on $\R^n$, we show that the scalar curvature cannot decay at a rate faster than $O(t^{-\frac{n}{2}})$. However in coordinates $R(g_t)$ can be written as
\begin{equation}
\label{scalar-curvature-coord}
R(g_t) = g_t^{-1}\ast g_t^{-1}\ast \nabla^2 h_t + g_t^{-1}\ast g_t^{-1}\ast g_t^{-1} \ast \nabla h_t \ast \nabla h_t
\end{equation}
and hence by the decay rates (\ref{decay-rates})
\begin{equation}
|R(g_t)| =O(t^{-\frac{n}{2p}-1}).
\end{equation}
Therefore any presence of scalar curvature in the case $p < \frac{n}{n-2}$ leads to a contradition, yielding a proof of theorem \ref{scalar-rigidity-thm}. In the borderline case $p=\frac{n}{n-2}$ we show that the $L^1$ norm of the scalar curvature $R(g_t)$ becomes instantly bounded at times $t>0$ (see lemma \ref{borderline}).

\section{Notation}
If not specified otherwise, we will take norms and derivatives with respect to the fixed Euclidean background metric $g_{eucl}$, e.g. $\nabla$ will denote the covariant derivative with respect to $g_{eucl}$.

For two metrics $g$ and $\overline{g}$ we say $g \leq \overline{g}$ if for all $\xi \in \R^n$ we have 
$$ g_{ij} \xi^i \xi^j \leq \overline{g}_{ij} \xi^i \xi^j.$$

We define a standard cut-off function $\eta$ on $\R^n$, such that $\eta, \nabla \sqrt{\eta} \in C^{\infty}_{c}(B_2(0))$, $0 \leq \eta \leq 1$, $\eta \equiv 1$ on $B_{1}(0)$ and $\norm{\nabla \eta} \leq c$ for some universal $c$. We also define a rescaled cut-off function $\eta_R(y) = \eta(\frac{y}{R})$ for $R>0$. By $\eta_{R,x}$ we will denote the rescaled cut-off function centered at $x$, defined by $\eta_{R,x} (y) = \eta (\frac{y-x}{R})$.
\section{Preliminaries}

\subsection{Ricci DeTurck flow}
Let $M$ be a manifolds with a fixed background metric $\overline{g}$. Then a family of metrics $(g_t)_{t \in I}$ is a solution of the Ricci DeTurck flow if it satisfies the following evolution equation
\begin{equation}
\label{RicciDeTurck}
\partial_t g_t = -2 \text{Ric}_{g_t} - \mathcal{L}_{X_{\overline{g}}(g_t-\overline{g})}(g_t),
\end{equation}
where $X_{\overline{g}} (h)$ is the Bianchi operator defined on symmetric 2-forms $h$ by
\begin{equation}
\label{X-eqn}
X_{\overline{g}}^i(h) = (\overline{g} + h)^{ij} (\overline{g} + h)^{pq} \left( - \nabla^{\overline{g}}_p h_{qj} + \frac{1}{2} \nabla^{\overline{g}}_j h_{pq}\right) \\
\end{equation}
and the covariant derivatives are taken with respect to $\overline{g}$. This equation is strongly parabolic, allowing standard parabolic theory to be applied to show short time existence on closed manifolds. We will only be considering Ricci DeTurck flow on $\R^n$ with the standard Euclidean metric $\overline{g} = g_{eucl}$ as the fixed background metric. Then the Ricci DeTurck equation takes the form (see \cite[Lemma 2.1]{Shi89} )
\begin{align}
\label{RdT}
\partial_t g_{ij} &= g^{ab} \nabla_a \nabla_b g_{ij}  \\
									& \quad + \frac{1}{2} g^{ab} g^{pq} \big( \nabla_i g_{pa} \nabla_j g_{qb} + 2 \nabla_a g_{jp} \nabla_q g_{ib} - 2\nabla_a g_{jp} \nabla_b g_{iq}\nonumber \\ 
									& \qquad \qquad \quad-2\nabla_j g_{pa} \nabla_b g_{iq} - 2 \nabla_i g_{pa} \nabla_b g_{jq} \big). \nonumber
\end{align}
It will be useful to consider the difference $h:= g_t - g_{eucl}$. Using the above equation we can express the evolution of $h$ as (see \cite[Equation (4.4)]{KochAndLamm}):
\begin{equation}
\label{RdT-simplified}
(\partial_t - \Delta)h = Q_0[h] + \nabla Q_1[h],
\end{equation}
where 
\begin{align}
\label{Q0}
Q_0[h] &= \frac{1}{2}(\delta + h)^{ab}(\delta + h)^{pq}\Big(\nabla_i h_{pa} \nabla_j h_{qb} + 2 \nabla_a h_{jp} \nabla_q h_{ib} - 2\nabla_a h_{jp} \nabla_b h_{iq} \\ \nonumber
&- 2\nabla_j h_{pa} \nabla_b h_{iq} - 2\nabla_i h_{pa} \nabla_b h_{jq}\Big) - \nabla_a(\delta +h)^{ab} \nabla_b h_{ij} 
\end{align}
and
\begin{equation}
\label{Q1}
\nabla Q_1[h] = \nabla_a\left( \left(\delta + h\right)^{ab}- \delta^{ab})\nabla_b h_{ij}\right).
\end{equation}

\subsection{Long time existence of Ricci DeTurck flow}
\label{section-KL}
In \cite[Theorem 4.3]{KochAndLamm} Koch and Lamm prove the existence of a long time weak solution to the Ricci DeTurck equation (\ref{RdT-simplified}). We will give a brief outline of their proof here. Using Duhamel's Principle the Ricci DeTurck equation (\ref{RdT-simplified}) with initial data $h_0$ can be written in the integral form
\begin{equation}
\label{weak-RdT}
h = F[h, h_0] = S[h_0] + V[h],
\end{equation}
where
\begin{align}
S[h_0](x,t) &= \int_{\R^n} K(x,t;y,0)h_0(y)\, \mathrm{d}y \\
V[h](x,t) &= \int_0^t\int_{\R^n} K(x,t;y,s)Q_0[h] + \nabla_x K(x,t;y,s) Q_1[h] \, \mathrm{d}y \, \mathrm{d}s
\end{align}
and 
\begin{equation}
K(x,t;y,s) = (4\pi (t-s))^{-\frac{n}{2}} \exp\left( - \frac{|x-y|^2}{4(t-s)}\right)
\end{equation} 
is the kernel of the linear heat equation $\partial_t u = \Delta u$ on $\R^n$. Koch and Lamm construct the Banach spaces $X_T$, $0 < T \leq \infty$, defined by
\begin{equation}
X_T = \{ h \Big| \norm{h}_{X_T} < \infty \}
\end{equation}
and
\begin{align}
\norm{h}_{X_T} &= \sup_{0<t<T} \norm{h(t)}_{L^{\infty}(\R^n)} \\ \nonumber
		& \quad + \sup_{x\in\R^n} \sup_{0<R^2<T} \left( R^{-\frac{n}{2}} \norm{\nabla h}_{L^2(B_R(x)\times(0,R^2))} + R^{\frac{2}{n+4}} \norm{\nabla h}_{L^{n+4}(B_R(x)\times(\frac{R^2}{2}, R^2))}\right).
\end{align}
They show that there exist constants $\epsilon = \epsilon(n)>0$ and $C=C(n)>0$ such that whenever
\begin{equation}
\norm{h_0}_{L^{\infty}(\R^n)}< \epsilon
\end{equation}
 $F[\cdot, h_0]$ is a contraction mapping on the subspace
\begin{equation}
 \{h \in X_{\infty} \Big| \norm{h}_{X_{\infty}} < C \norm{h_0}_{L^{\infty}(\R^n)} \} \subset X_{\infty}.
\end{equation}
This proves the existence of a solution $h_t \in X_{\infty}$ with $\norm{h_t}_{X_{\infty}} < C \norm{h_0}_{L^{\infty}(\R^n)}$ to the weak Ricci DeTurck equation (\ref{weak-RdT}) for $L^{\infty}$ small initial data. They also prove that such a weak solution $h_t$ is analytic on $\R^n\times(0,\infty)$ and that the following decay rates hold for every multi-index $\alpha \in \N^n_0$ and $k\in \N_{0}$
\begin{equation}
\label{KL-decay-rates}
\sup_{x\in \R^n}\sup_{t >0} | (t^{\frac{1}{2}}\nabla)^{\alpha}(t\partial_t)^k h(x,t)| \leq c \norm{h_0}_{L^{\infty}(\R^n)}R^{|\alpha|+k}(|\alpha| + k)!.
\end{equation}
Here $R=R(n)$, $c = c(n)$ are constants depending on $n$ only. Finally, they construct an analytical operator 
\begin{equation}
A: B_{\epsilon}(0) \subset C^{0}(\R^n) \longrightarrow X_{\infty}
\end{equation}
for sufficiently small $\epsilon=\epsilon(n)>0$ such that $A(h_0) \in X_{\infty}$ is the weak solution to (\ref{weak-RdT}) for $\norm{h_0}_{L^{\infty}(\R^n)} < \epsilon$. 

Because the operator $A$ is analytic, we can choose $\epsilon>0$ sufficiently small such that 
\begin{equation}
\norm{A(h^1_0)-A(h^2_0)}_{X_{\infty}} \leq C_1 \norm{h^1_0-h^2_0}_{C^{0}(\R^n)}
\end{equation}
for $h^1_0, h^2_0 \in B_{\epsilon}(0) \subset C^0(R^n)$ and $C_1 = C_1(\epsilon)$ a constant. Inspecting the norm on $X_{\infty}$ we then see that
\begin{equation}
\label{l-infty-estimate}
\sup_{0<t<\infty} \norm{h^1_t - h^2_t}_{L^{\infty}(\R^n)} \leq C_1 \norm{h^1_0-h^2_0}_{L^{\infty}(\R^n)},
\end{equation}
where $h^i_t = A(h^i_0)(\cdot, t)$ for $i = 1,2$.

In this paper we will prove regularity of the weak solution $h_t$ up to the boundary $t=0$, i.e. for smooth initial data $h_0$ the weak solution $h_t$ and its spatial derivatives are continuous on $\R^n\times[0,\infty)$.

\subsection{Ricci flow and long time existence}
\label{RicciFlow}
A family of time dependent metrics $(\tilde{g}_t)_{t\geq0}$ is a solution to the Ricci flow equation if
\begin{equation}
\label{RF}
\partial_t \tilde{g}_t = -2 Ric(\tilde{g}_t), \quad t \geq 0.
\end{equation}
Ricci and Ricci DeTurck flow are related by a family of time dependent diffeomorphisms (see \cite{Topping} for details): Assume that $(g_t)_{t\geq 0}$ is a solution to the Ricci DeTurck equation (\ref{RdT}) and the family of time dependent diffeomorphisms $\Phi_t$ generated by the vector field $X(g_t)$ (see (\ref{X-eqn}))
\begin{align}
\label{Phi-eqn}
\partial_t \Phi_t &= X_{\overline{g}}(g_t-\overline{g})(\Phi_t), \quad t\geq0 \\
			\Phi_0 &= id 
\end{align}
exists, then the pullback $\tilde{g}_t = \Phi_t^{\ast}g_t$ is a solution to the Ricci flow equation (\ref{RF}). In \cite[Lemma 9.1]{MilesSimon} it was shown that for perturbations $g_t \in C^{\infty}(\R^n \times [0,\infty))$ of the Euclidean metric on $\R^n$ evolving under Ricci DeTurck flow this family of diffeomorphisms $\Phi_t \in C^{\infty}(\R^n \times [0,\infty))$ exists as long as the decay rates (\ref{KL-decay-rates}) hold. 

\section{Regularity of weak solution}
In this section we will study the behavior at $t=0$ of the weak solution $h_t$ to the Ricci DeTurck equation constructed in \cite[Theorem 4.3]{KochAndLamm}. Our main goal will be to prove the following theorem:
\begin{thm}
\label{RdT-smooth}
There exists an $\epsilon = \epsilon(n)>0$ such that for smooth initial data $h_0\in C^{\infty}(\R^n)$ satisfying $\norm{h_0}_{L^{\infty}(\R^n)} \leq \epsilon$, a classical smooth solution $h_t \in C^{\infty}(\R^n \times [0,\infty))$ to the Ricci DeTurck equation (\ref{RdT-simplified}) exists. 
\end{thm}

The first step is to prove that for smooth initial data of compact support the weak solution is classical and smooth. Approximating our intial data by compactly supported functions, we then prove $C^{2,1}$ regularity of $h_t$. The main difficulty at this step will be to prove a local boundary estimate using standard parabolic Hölder estimates, which allows us to pass to the limit via Arzelà–Ascoli and a diagonal argument. In the final step we will bootstrap our boundary estimate to show that the solution $h_t$ and its derivatives are continuous up to $t=0$.

Below we begin by proving the compact case.

\begin{lem}
\label{compact-case-smooth}
There exists an $\epsilon=\epsilon(n)>0$ such that for $h_0 \in C^{\infty}_c(\R^n)$ satisfying $\norm{h_0}_{L^{\infty}(\R^n)}< \epsilon$ the weak solution $h_t$ is classical and smooth on $\R^n \times [0,\infty)$.
\end{lem}
\begin{proof}
For initial data $h_0\in C^{\infty}_c(\R^n)$ the curvature of the resulting metric $g_0 = g_{eucl} + h_0$ is bounded. Therefore by \cite[Theorem 4.3]{Shi89} there exists a smooth solution $g_t$, $0\leq t\leq T$, to the Ricci DeTurck flow, where $T >0$ depends on $n$ and the curvature bound. Furthermore by choosing $T$ sufficiently small (\cite[Theorem 2.5 \& Lemma 4.1]{Shi89}) we can ensure that $g_t - g_{eucl} \in X_{T}$. Since the weak solution $h_t$ to the Ricci DeTurck equation (\ref{weak-RdT}) is unique by \cite[Theorem 4.3]{KochAndLamm}, it is identical to $g_t - g_{eucl}$ on $(0,T) \times \R^n$. As the weak solution $h_t$ is analytic on $(0,\infty) \times \R^n$, we deduce the desired result.
\end{proof}

We now prove an a priori local boundary Hölder estimate for $h_t$. Our argument closely follows \cite[Prof. 2.5]{RB1}, where the corresponding interior estimates were derived. This estimate will allow us to approximate non-compact, locally Hölder-regular initial data by a sequence of compactly supported smooth initial data and show that the corresponding flows converge in a local Hölder sense.

Let us fix some notation first. If $\Omega \subset \R^n\times\R$ is some parabolic neighbourhood (e.g. $\Omega = B_r(0)\times(0,r^2))$) we denote by $C^{2m,m}(\Omega)$ the space of functions differentiable $i$ times in the spatial direction and $j$ times in the time direction as long as $i+2j \leq 2m$. For $\alpha\in(0,\frac{1}{2})$ the corresponding Hölder space will be denoted by $C^{2m,2\alpha; m , \alpha}(\Omega)$. Greek letters will always refer to the Hölder exponent. In the following we will use weighted Hölder norms on $C^{2m,2\alpha; m , \alpha}(\Omega)$, which are invariant under parabolic dilations: Assume 
\begin{equation}
r_{\Omega} = \min(r : \Omega \subset B_r(p)\times[t-r^2, t] \text{ for some } p, t) < \infty
\end{equation}
then
\begin{equation}
 \norm{u}_{C^{2m,2\alpha;m,\alpha}(\Omega)} = \sum_{|\iota|+2k \leq 2m} r_{\Omega}^{|\iota| + 2k} \left( \norm{\nabla^{\iota}\partial_t^k u}_{C^{0}} + r_{\Omega}^{2\alpha}[\nabla^{\iota} \partial_t^k u]_{2\alpha,\alpha}\right)
\end{equation}
When $m=0$ we will write $\norm{u}_{C^{2\alpha;\alpha}(\Omega)} := \norm{u}_{C^{0,2\alpha;0,\alpha}(\Omega)}$ for brevity.

By $L$ we will denote a second order elliptic operator with real coefficients and constant of ellipticity $\kappa>0$ of the form $L(x,t) = a^{ij}(x,t)\partial_i\partial_j + b^{i}(x,t)\partial_i + c(x,t)$.

For $z = (x,t) \in \R^{n+1}$ we will denote by $\Omega_r(z)$ the parabolic neighborhood $B_r(x)\times(t-r^2, t) \subset \R^{n+1}$.

Now we recall an a priori estimate for the linear parabolic initial value problem.
\begin{lem}
\label{parabolic-estimate}
Let $\Omega = \R^n\times(0,\infty)$ and fix $z=(x_0,t_0)\in \R^{n+1}$. Assume that $\norm{a,b,c}_{C^{2\alpha;\alpha}(\Omega_{2r}(z)\cap \Omega)} \leq K$, $u \in C^{2,2\alpha;1,\alpha}(\Omega_{2r}(z) \cap \Omega )$, $f \in C^{2\alpha;\alpha}(\Omega_{2r}(z) \cap \Omega)$ and $u_0 \in C^{2,\alpha}(B_{2r}(x_0))$ . If $u$ solves the parabolic equation
\begin{align}
\partial_t u - Lu &= f, \qquad \text{on } \Omega_{2r}(z)\cap \Omega, \\
					u(\cdot, 0) &= u_0, \qquad \text{on } \Omega_{2r}(z)\cap \partial \Omega,
\end{align}
then we have the boundary estimate
\begin{align}
\norm{u}_{C^{2,2\alpha;1,\alpha}(\Omega_{r}(z)\cap \Omega)} &\leq C \Big(r^2 \norm{f}_{C^{2\alpha;\alpha}(\Omega_{2r}(z)\cap \Omega)} \\
\nonumber &\qquad \qquad + \norm{u}_{C^0(\Omega_{2r}(z)\cap \Omega)} + \norm{u_0}_{C^{2,2\alpha}(B_{2r}(x_0))} \Big),
\end{align}
where $C$ is a constant that depends on $\alpha$, $n$, $\kappa$ and $K$.
\end{lem}

\begin{proof}
By scale invariance we can assume without loss of generality $r=1$. Then the result follows from \cite[Exercise 9.2.5]{KYL} applied to $u-u_0$ and  \cite[Remark 8.11.2]{KYL}.
\end{proof}

We now use above lemma \ref{parabolic-estimate} to prove the main estimate of this section.
\begin{lem}
\label{localisation-estimate}
Fix $r>0$ and $\alpha \in (0, \frac{1}{2})$. Consider the parabolic neighborhoods $\Omega_r = B_r \times (0,r^2) \subset \Omega'_{2r} = B_{2r} \times (0, r^2)$ and assume that $u \in C^{2,2\alpha; 1,\alpha}(\Omega'_{2r})$ satisfies the equation
\begin{align}
\label{parabolic-eqn}
(\partial_t - \Delta)u &= Q[u] = f_1 (r^{-1} x, u) \cdot \nabla u \otimes \nabla u + f_2(r^{-1}x, u)\cdot u\otimes \nabla^2 u \\
u(\cdot, 0) & = u_0,
\end{align}
where $f_1$, $f_2$ are smooth vector-valued functions in $x$ and $u$ and $f_1$, $f_2$ can be paired with tensors $\nabla \otimes \nabla$ and $u \otimes \nabla^2 u$ respectively.
Then there are constants $\epsilon'>0$ and $C'< \infty$ depending only on $\alpha$, $n$ and the $f_i$ such that if 
\begin{equation}
H = \norm{u}_{L^{\infty}(\Omega'_{2r})} + \norm{u_0}_{C^{2,2\alpha}(B_{2r})}  < \epsilon',
\end{equation}
then
\begin{equation}
\label{localisation-estimate-inequality}
\norm{u}_{C^{2,2\alpha; 1, \alpha}(\Omega_r)} < C' H.
\end{equation}
Moreover, the lemma still holds if u is vector-valued.
\end{lem}

\begin{proof}
We will adapt the proof of \cite[Prop. 2.5]{RB1} to the initial value problem. We begin by introducing a new weighted norm for $0<\theta\leq 1$ and $\Omega$ a parabolic neighbourhood:
\begin{equation}
\norm{u}^{(\theta)}_{C^{2m,2\alpha; m, \alpha}(\Omega)} = \sum_{|\iota|+2k \leq 2m} (r_{\Omega} \theta)^{|\iota|+2k} \left( \norm{\nabla^{\iota} \partial_t^k u}_{C^{0}} + (r_{\Omega}\theta)^{2\alpha} [\nabla^{\iota} \partial_t^k u]_{2\alpha,\alpha} \right),
\end{equation}
For $\theta = 1$, this norm agrees with the norm defined above. We also introduce a weighted norm on $C^{2m,2\alpha}(\R^n)$ in an analogous way. Choosing any $z=(x,t)\in R^n \times (0,r^2)$ such that $B_{\theta r}(x) \subset B_r$ and applying lemma \ref{parabolic-estimate} on $\Omega_{\theta r}(z)$ we deduce
\vspace{0.3cm}
\begin{indentpar}{1cm}
\textit{Assume we are in the setting of lemma \ref{parabolic-estimate}, then}
\begin{align}
\label{weighted-parabolic-estimate}
\norm{u}^{(\theta)}_{C^{2,2\alpha; 1, \alpha}(\Omega_r)}&\leq C \Big( (r\theta)^2 \norm{f}^{(\theta)}_{C^{2\alpha;\alpha}(B_{(1+\theta)r}\times(0, r^2))} \\ \nonumber
& \qquad\qquad+ \norm{u}^{(\theta)}_{C^0(B_{(1+\theta)}r\times(0, r^2))} + \norm{u_0}^{(\theta)}_{C^{2,2\alpha}(B_{(1+\theta)r})} \Big)
\end{align}
\end{indentpar}
\vspace{0.3cm}
In the following we may assume by scaling invariance that $r = 1$ and we will abbreviate by $C$ any constant that depends on $n$, $\alpha$ and $f_i$, $i= 1, 2$.
Set 
\begin{align}
r_k &= \sum_{i=0}^k 2^{-i} = 2 - 2^{-k}, \\
\theta_k &= \frac{r_{k+1}}{r_k} -1, \\
\Omega_k &= B_{r_k}(0) \times [0, 1].
\end{align} 
By (\ref{weighted-parabolic-estimate}) for $r = r_k$ we have
\begin{equation}
a_k := \norm{u}^{(\theta_k)}_{C^{2,2\alpha; 1, \alpha}(\Omega_k)} \leq C \left(\theta_k^2 \norm{Q[u]}_{C^{2\alpha;\alpha}(\Omega_{k+1})} + H \right).
\end{equation}
Observe that since $\theta_k \rightarrow 0$, we have $a_k \rightarrow a_{\infty} = \norm{u}_{C^0(\Omega')} \leq H$. We now estimate $Q[u]$ in terms of $u$ using (\ref{parabolic-eqn}). For this note that for $i = 1, 2$ 
\begin{equation}
\norm{f_i(x,u)}^{(\theta_{k+1})}_{C^{2\alpha;\alpha}(\Omega_{k+1})} \leq C\left(1 + \norm{u}^{(\theta_{k+1})}_{C^{2\alpha;\alpha}(\Omega_{k+1})}\right) \leq C(1+a_{k+1})
\end{equation}
So we obtain
\begin{align}\nonumber
\norm{f_1 \cdot \nabla u \otimes \nabla u}^{(\theta_{k+1})}_{C^{2\alpha;\alpha}(\Omega_{k+1})} &\leq \norm{f_1}^{(\theta_{k+1})}_{C^{2\alpha;\alpha}(\Omega_{k+1})} \left(\norm{\nabla u}^{(\theta_{k+1})}_{C^{2\alpha;\alpha}(\Omega_{k+1})}\right)^2 \\
&\leq C \theta^{-2}_{k+1} \left( a_{k+1}^2 + a_{k+1}^{3}\right)
\end{align}
Similarly we have
\begin{equation}
\norm{f_2 \cdot u \otimes \nabla^2 u}^{(\theta_{k+1})}_{C^{2\alpha;\alpha}(\Omega_{k+1})} \leq C \theta^{-2}_{k+1} \left( a_{k+1}^2 + a_{k+1}^{3}\right).
\end{equation}
We conclude
\begin{equation}
\norm{Q[u]}^{(\theta_{k+1})}_{C^{2\alpha;\alpha}(\Omega_{k+1})} \leq C \theta^{-2}_{k+1} \left( a_{k+1}^2 + a_{k+1}^{3}\right).
\end{equation}
Hence
\begin{equation}
a_{k} \leq C(a_{k+1}^2 + a_{k+1}^{3} + H)
\end{equation}
Therefore the quantity $b_k = \frac{a_k}{H}$ satisfies the inequality
\begin{equation}
b_{k} \leq C(H b_{k+1}^2 + H^{2} b_{k+1}^{3} + 1).
\end{equation}
Assuming without loss of generality that $C>1$ and setting $\epsilon' = \frac{1}{16C^2}$ we see that
\begin{equation}
b_{k} \leq \frac{b_{k+1}^2}{16 C} + \frac{b_{k+1}^{3}}{16^2 C^{3}} + C,
\end{equation}
because we assumed that $H\leq \epsilon'$. From this we see that if $b_{k+1} \leq 2C$ then $b_{k} \leq 2C$ as well. Since $b_k \rightarrow \frac{a_{\infty}}{H} \leq 1 \leq 2C$ as $k \rightarrow \infty$ we conclude that $b_0 \leq 2C$ and thus $a_0 \leq 2CH$.
\end{proof}

Using above lemma we can now proceed to proving $C^{2,2\alpha; 1, \alpha}$-regularity. 

 \begin{thm}
\label{regularity-RdT}
Fix an $\alpha \in (0, \frac{1}{2})$. Then there exists an $\epsilon>0$ such that for any $h_0 \in C^{2,2\alpha}_{loc}(\R^n)$ and $\norm{h_0}_{L^{\infty}(\R^n)} \leq \epsilon$ the weak solution $h_t$ of the Ricci DeTurck equation (\ref{RdT-simplified}) constructed in \cite[Theorem 4.3]{KochAndLamm} is classical. In particular for any relatively compact open set $\Omega \subset\subset \R^n\times [0,\infty)$ we have $h_t \in C^{2,2\alpha; 1, \alpha}(\Omega)$.
\end{thm}

\begin{proof}[Proof of the Theorem \ref{regularity-RdT}]
Choose $\epsilon'>0$ as in lemma \ref{localisation-estimate} above. Then choose $\epsilon>0$ such that a weak solution $h_t$ to the Ricci DeTurck equation (\ref{RdT-simplified}) exists for initial data $h_0$ satisfying $\norm{h_0}_{L^{\infty}(\R^n)} < \epsilon$.  By choosing $\epsilon$ sufficiently small we may assume that $\norm{h_t}_{L^{\infty}(\R^n)}\leq \frac{1}{3}\epsilon'$ uniformly for $t\geq 0$ (see \cite[theorem 4.3]{KochAndLamm}). Take a sequence $h^{i}_0 \in C^{\infty}_c(\R^n)$, $i = 1, 2, 3, \cdots$ satisfying $\norm{h^i_0}_{L^{\infty}(\R^n)} \leq \epsilon$, which locally converges to $h_0$ in $C^{2,2\alpha}_{loc}(\R^n)$ as $i \rightarrow \infty$. By lemma \ref{compact-case-smooth} a smooth solution $h^i_t$ to the Ricci DeTurck equation (\ref{RdT-simplified}) with initial data $h^i_0$ exists for all $i\in\N$. For each $x\in\R^n$ we can then choose an $r_x>0$ such that the quantity $H$ of lemma \ref{localisation-estimate} is less than $\frac{5}{6}\epsilon'$. Thus we can choose countable many points $(x_i)_{i\in \N} \in \R^n$ such that $\R^n\times \{t=0\}$ is covered by sets of the form $\Omega_i = B_{r_{x_i}}(x_i) \times [0, r_{x_i}^2)$ on which (\ref{localisation-estimate-inequality}) holds for all $h^i_0$ satisfying $\norm{h^i_0 - h_0}_{C^{2,2\alpha}(\Omega_i)} \leq \frac{1}{6} \epsilon'$. Note that by \cite[theorem 4.3]{KochAndLamm} we know that on any compactly supported open set $\Omega \subset\subset \R^n\times (0,\infty)$ we have uniform bounds of the form $\norm{h^i_t}_{C^{2,2\alpha;1,\alpha}(\Omega)} \leq C(\Omega)$. Thus, by Arzela-Ascoli and a diagonal argument we can pass to a limit giving us the desired result.
\end{proof}

\begin{remark}
Using similar arguments we can show that if we start with continuous initial data $h_0\in C^{0}(\R^n)$, the corresponding weak solution is continuous up to the boundary: Take $h^i_0 \in C^{2,2\alpha}(\R^n)$ converging uniformly to $h_0$ in $C^0(\R^n)$ and use the estimate (\ref{l-infty-estimate}) to prove convergence.
\end{remark}

Now we can prove theorem \ref{RdT-smooth}.

\begin{proof}[Proof of theorem \ref{RdT-smooth}]
Choose $\epsilon>0$ such that above theorem \ref{regularity-RdT} applies and thus $h\in C^{2,2\alpha;1,\alpha}_{loc}(\R^n\times [0,\infty))$. We can then write the Ricci DeTurck equation (\ref{RdT}) in the form 
\begin{equation}
\partial_t h - (g_{eucl} + h)^{ab}\nabla_a \nabla_b h = f(x,h) \ast \nabla h \ast \nabla h,
\end{equation}
for $f$ a smooth vector-valued function. The upshot is that the righthand side is of lower order than the lefthand side, allowing us to bootstrap \cite[theorem 8.12.1 \& exercise 8.12.4]{KYL} to prove that $h$ is smooth.
\end{proof}

\begin{remark}
Strictly speaking \cite[theorem 8.12.1 \& exercise 8.12.4]{KYL} only apply to proving interior regularity. However, one can check using lemma \ref{parabolic-estimate} that they can easily be adapted to hold in the boundary case.
\end{remark}

\section{$L^p$ bounds and $L^{\infty}$-decay rates for Ricci DeTurck flow}
\subsection{$L^p$ bounds}
In the following we will show that a solution $h_t$ to the Ricci DeTurck equations with initial data $h_0$  bounded in $L^p(\R^n)$ for some $1\leq p < \infty$ remains bounded in $L^p$, provided the initial data is sufficiently small in $L^{\infty}(\R^n)$. In particular we will prove the following theorem:
\begin{thm}
\label{main-lp-thm}
There exist constants $\tilde{\epsilon} = \tilde{\epsilon}(n)>0$ and $\tilde{C} = \tilde{C}(n)>0$ such that the following holds: Given a solution $h_t\in C^{2,1}(\R^n\times[0,\infty))$ to the Ricci DeTurck equation (\ref{RdT-simplified}) satisfying (i) $\norm{h_t}_{L^{\infty}(\R^n\times(0,\infty))}< \tilde{\epsilon}$ and (ii) $\norm{h_0}_{L^{p}(\R^n)}< \infty$ for some $1\leq p < \infty$ the estimate
\begin{equation}
\norm{h_t}_{L^p(\R^n)} \leq \tilde{C} \norm{h_0}_{L^p(\R^n)}, \quad t \geq 0,
\end{equation}
holds. If $p\geq 2$ we can take $\tilde{C}=1$.
\end{thm}
\begin{remark}
\hspace{0.5cm}
\begin{enumerate}
\item
By estimate (\ref{KL-decay-rates}) we can relax condition (i) to $\norm{h_0}_{L^{\infty}(\R^n))}< \tilde{\epsilon}$ by choosing $\tilde{\epsilon}$ sufficiently small.
\item
The result agrees with the analogue for the linear heat equation on $\R^n$. 
\end{enumerate}

\end{remark}

Our strategy will be to find local $L^p$ estimates for the solution $h_t$ and integrate these to global ones using Grönwall's Lemma. For technical reasons we will first prove the result for $2\leq p < \infty$. This will yield an $L^2(\R^n\times (0, \infty))$ bound on $\nabla h$, which will allow us to generalize to the $1\leq p <2$ case. Below we state our first local $L^p$ bound.

\begin{lem}
\label{local-lp-bound-p-greater-2}
There exist constants $\tilde{\epsilon}>0$ and $C_1>0$ such that for a solution $h \in C^{2,1}(\R^n \times [0,\infty))$ with $\norm{h}_{L^{\infty}(\R^n \times [0,\infty))} <  \tilde{\epsilon}$ to the Ricci DeTurck equation (\ref{RdT-simplified}) the estimate 
\begin{equation}
\label{estimate-p-greater-2}
\partial_t \int_{\R^n} \eta_R |h|^p\, \mathrm{d}x \leq \frac{C_1(1+p)}{R^2} \int_{B_{2R}\setminus B_{R}} |h|^p\, \mathrm{d}x.
\end{equation}
holds in the barrier sense for $2 \leq p < \infty$.
\end{lem}
\begin{proof}
In the following we will denote by $C$ any constant depending on $n$ and $\eta$ only. We will also assume $\tilde{\epsilon} < \frac{1}{2}$. Let $\delta > 0$ and define
\begin{equation*}
|h|_{\delta} = \sqrt{\langle h, h \rangle + \delta}
\end{equation*} 
Then $|h|_{\delta} \in C^{2,1}(\R^n \times (0,\infty))$, $|h|_{\delta}\geq\delta$, $|h|_{\delta} > |h|$ and $|h|_{\delta} \rightarrow |h|$ pointwise as $\delta \rightarrow 0$. A simple calculation shows
\begin{equation}
\label{laplace-of-pnorm}
\Delta |h|_{\delta}^p = p(p-2)|h|_{\delta}^{p-4} \langle \nabla_a h, h \rangle \langle \nabla_a h, h \rangle + p |h|_{\delta}^{p-2} \left( \langle \nabla h, \nabla h \rangle + \langle \Delta h, h \rangle \right),
\end{equation}
where we sum over equal indices. Since $h$ satisfies the Ricci DeTurck equation (\ref{RdT-simplified}) we obtain
\begin{align}
 \partial_t |h|_{\delta} ^p  &= p |h|_{\delta}^{p-2} \langle \Delta h, h \rangle + p |h|_{\delta}^{p-2}\langle Q[h], h \rangle \\ 
 							  \label{inequ} &= \Delta |h|_{\delta}^p - p|h|_{\delta}^{p-2} |\nabla h|^2 + p(2-p)|h|_{\delta}^{p-4}\langle \nabla_a h, h \rangle \langle \nabla_a h, h \rangle 
 							  \\ \nonumber & \qquad\qquad+ p |h|_{\delta}^{p-2} \langle Q_0[h] + \nabla Q_1[h] , h \rangle,
 \end{align}
Therefore multiplying (\ref{inequ}) by $\eta_R$ and integrating by parts over $\R^n$, we obtain 
\begin{align}
\partial_t \int_{\R^n} \eta_R |h|^p_{\delta} \, \mathrm{d}x &= \int_{\R^n} \Delta \eta_R |h|^p_{\delta}  - \eta_R p|h|_{\delta}^{p-2} |\nabla h|^2 + \eta_R p(2-p)|h|_{\delta}^{p-4}\langle \nabla_a h, h \rangle \langle \nabla_a h, h \rangle \\
\nonumber& \qquad\qquad + p\eta_R |h|_{\delta}^{p-2} \langle Q_0[h] + \nabla Q_1[h] , h \rangle  \, \mathrm{d}x.
\end{align}
Recalling the definitions of $Q_0[h]$ and $Q_1[h]$ in (\ref{Q0}) and (\ref{Q1}) respectively, we see that
\begin{align}
Q_0[h] &= (g_{eucl} + h)^{-1}\ast(g_{eucl} + h)^{-1}\ast \nabla h \ast \nabla h \\
Q_1[h]^a_{ij} &= \left( (g_{eucl} + h)^{ab} - g_{eucl}^{ab}\right) \nabla_b h_{ij}
\end{align}
By assumption $\norm{h}_{L^{\infty}(\R^n \times [0,T])} <  \frac{1}{2}$, so 
\begin{align}
\label{Q0-est}|Q_0[h]| &\leq C |\nabla h|^2 \\
\label{Q1-est}|Q_1[h]| &\leq C |h||\nabla h| 
\end{align}
Thus by Cauchy-Schwarz and integration by parts we have
\begin{align}
\partial_t \int_{\R^n} \eta_R |h|^p_{\delta} \, \mathrm{d}x &\leq \int_{\R^n} \Delta \eta_R |h|^p_{\delta} +p(C\tilde{\epsilon}-1) \eta_R p|h|_{\delta}^{p-2} |\nabla h|^2 + \eta_R p(2-p)|h|_{\delta}^{p-4}\langle \nabla_a h, h \rangle \langle \nabla_a h, h \rangle \\
\nonumber& \qquad\qquad - p  Q_1[h]^a_{ij}\nabla_a \left(h_{ij} \eta_R |h|_{\delta}^{p-2} \right )  \, \mathrm{d}x
\end{align}
Focussing on the last term in the integrand above we have 
\begin{align}
\int_{\R^n} -pQ_1[h]^a_{ij}\nabla_a \left(h_{ij} \eta_R |h|_{\delta}^{p-2} \right )  \, \mathrm{d}x &\leq \int_{\R^n} p|Q_1[h]| \left( |\nabla h| \eta_R |h|_{\delta}^{p-2}  + |h| |\nabla \eta_R| |h|_{\delta}^{p-2} \right) \\
\nonumber& \qquad\qquad  - pQ_1[h]^a_{ij} h_{ij} \eta_R \nabla_a |h|_{\delta}^{p-2}\, \mathrm{d}x  \\
\nonumber&\leq \int_{\R^n} Cp\eta_R|h||h|_{\delta}^{p-2}|\nabla h|^2  + Cp|\nabla \eta_R| |h|^2|h|_{\delta}^{p-2}|\nabla h| \\
\label{reason-p-greater-2}& \qquad  - \eta_R p(p-2) |h|_{\delta}^{p-4} \langle \nabla_a h , h \rangle \left( (g_{eucl} + h)^{ab} - g_{eucl}^{ab}\right) \langle \nabla_b h , h \rangle \, \mathrm{d}x \\
\nonumber& \leq \int_{\R^n} Cp\tilde{\epsilon}\eta_R|h|_{\delta}^{p-2}|\nabla h|^2  + Cp|\nabla \eta_R| |h|^2|h|_{\delta}^{p-2}|\nabla h| \\
\nonumber& \qquad  + \eta_R p(p-2) |h|_{\delta}^{p-4} \langle \nabla_a h , h \rangle \langle \nabla_a h , h \rangle \, \mathrm{d}x \\
\nonumber&\leq \int_{\R^n} Cp\tilde{\epsilon}\eta_R|h|_{\delta}^{p-2}|\nabla h|^2  + Cp\tilde{\epsilon}\left(\nabla \sqrt{\eta_R}\right)^2|h|^{p}_{\delta} \\
\nonumber& \qquad  + \eta_R p(p-2) |h|_{\delta}^{p-4} \langle \nabla_a h , h \rangle \langle \nabla_a h , h \rangle \, \mathrm{d}x, 
\end{align}
where in the last step we used Young's inequality to estimate
\begin{align}
|\nabla \eta_R| |h|^2|h|_{\delta}^{p-2}|\nabla h| &= \eta_R^{\frac{1}{2}} |\nabla h| |h|^{\frac{1}{2}} |h|^{\frac{p-2}{2}}_{\delta} \cdot \left(\frac{|\nabla \eta_R|}{\eta_R^{\frac{1}{2}}}\right) |h|^{\frac{3}{2}} |h|^{\frac{p-2}{2}}_{\delta} \\
\nonumber&\leq \frac{1}{2} \eta_R|\nabla h|^2 |h| |h|^{p-2}_{\delta} + \frac{1}{2}\left(2 \nabla \sqrt{\eta_R}\right)^2|h|^3|h|^{p-2}_{\delta} \\
\nonumber&\leq \frac{\tilde{\epsilon}}{2} \eta_R|\nabla h|^2|h|^{p-2}_{\delta} + 2\tilde{\epsilon}\left(\nabla \sqrt{\eta_R}\right)^2|h|^{p}_{\delta} 
\end{align}
Combining above inequalities we obtain 
\begin{align}
\partial_t \int_{\R^n} \eta_R |h|^p_{\delta} \, \mathrm{d}x &\leq \int_{\R^n} \left( \Delta \eta_R + C \tilde{\epsilon} p\left(\nabla \sqrt{\eta_R}\right)^2 \right) |h|^{p}_{\delta} +p(C\tilde{\epsilon}-1) \eta_R |h|_{\delta}^{p-2} |\nabla h|^2 \, \mathrm{d}x \\
\nonumber& \leq \frac{C(1+p)}{R^2} \int_{B_{2R}\setminus B_{R}} |h|^{p}_{\delta} \mathrm{d}x 
\end{align}
for $\tilde{\epsilon} < \frac{1}{C}$. After integrating with respect to $t$ and taking the limit $\delta \rightarrow 0$, we see that the desired inequality holds true. 
\end{proof}

\begin{remark}
The reason why above lemma only holds for $2\leq p<\infty$ is because the term $ - \eta_R p(p-2) |h|_{\tilde{\epsilon}}^{p-4} \langle \nabla_a h , h \rangle (g_{eucl} + h)^{ab} \langle \nabla_b h , h \rangle$ in line (\ref{reason-p-greater-2}) is negative only for $p\geq2$.
\end{remark}

Now we can use Grönwall's Lemma to prove our theorem in the $p\geq 2$ case. 

\begin{proof}[Proof of theorem \ref{main-lp-thm} for $p\geq 2$]
Choose $\tilde{\epsilon}>0$ and $C_1$ such that lemma \ref{local-lp-bound-p-greater-2} holds and consider the following quantity
\begin{equation}
A(t,R) = \sup_{x \in \R^n} \int_{\R^n} \eta_{R,x}(y) h^p(y,t) \, \mathrm{d}y.
\end{equation}
Centering (\ref{estimate-p-greater-2}) at $x$ and integrating with respect to time we obtain
\begin{equation}
\int_{\R^n} \eta_{R,x} h^p(y,t) \,\mathrm{d}y \leq \int_{\R^n} \eta_{R,x} h^p(y,0) \, \mathrm{d}y + \frac{C_1(1+p)}{R^2}\int_0^t \int_{B_{2R}(x) \setminus B_{R}(x)} h^p (y,t) \, \mathrm{d}y \, \mathrm{d}t. 
\end{equation}
Because we can cover $B_{2R}(x) \setminus B_{R}(x)$ by a finite number $N$ of balls of radius $R$ and furthermore this number only depends on the dimension $n$, we obtain
\begin{equation}
A(t,R) \leq A(0,R) + \frac{NC_1(1+p)}{R^2} \int_0^t A(t,R) \, \mathrm{d}t.
\end{equation}
By Grönwall's inequality (see appendix) we deduce
\begin{equation}
A(t,R) \leq A(0,R) \exp\left(\frac{NC_1(1+p)}{R^2} t\right).
\end{equation}
Taking the limit $R \rightarrow \infty$ shows
\begin{equation}
\norm{ h(\cdot, t) }_{L^p(\R^n)} \leq \norm{ h(\cdot, 0) }_{L^p(\R^n)}
\end{equation}
\end{proof}

Below we generalize to the $1\leq p < 2$ case. For this we will need the following $L^p$ estimate, which is analog to the one stated in lemma \ref{local-lp-bound-p-greater-2}.

\begin{lem}
\label{local-lp-bound-p-less-2}
There exists a $C_2>0$ such that for a solution $h \in C^{2,1}(\R^n \times [0,\infty))$ with $\norm{h}_{L^{\infty}(\R^n \times [0,\infty))} <  \frac{1}{2}$ to the Ricci DeTurck equation (\ref{RdT-simplified}) the following estimate holds in the barrier sense for $1 \leq p \leq 2 $ 
\begin{equation}
\label{lp-less-2-estimate}
\partial_t \int_{\R^n} \eta_R |h|^p \leq \frac{C_2}{R^2} \int_{B_{2R}\setminus B_{R}} |h|^p \, \mathrm{d}x + C_2 \int_{\R^n} \eta_R|\nabla h|^2 \,\mathrm{d}x,
\end{equation}
where $C_2$ depends on $n$ and $\eta$ only.
\end{lem}
\begin{remark}
The main difficulty in proving theorem \ref{main-lp-thm} for $1\leq p <2$ will be to bound $|\nabla h|^2$ in the above estimate.
\end{remark}

\begin{proof} Applying Cauchy-Schwarz to (\ref{inequ}) we obtain 
\begin{align}
 \partial_t |h|_{\delta} ^p  &\leq \Delta |h|_{\delta}^p +  p |h|_{\delta}^{p-2} \langle Q[h], h \rangle + | \nabla h|^2|h|_{\delta}^{p-4} \left ( p(2-p)|h|^2-p |h|_{\delta}^2 \right) \\
 							   &\leq \Delta |h|_{\delta}^p + p |h|_{\delta}^{p-2} \langle Q[h], h \rangle. 
\end{align}
Multiplying by the cut-off function $\eta_R$, integrating by parts and discarding boundary terms, we deduce
\begin{align}
\partial_t \int_{\R^n} \eta_R|h|_{\delta}^p &\leq \int_{\R^n} \left( \Delta \eta_R \right)|h|_{\delta}^p +  \eta_R p |h|_{\delta}^{p-2} \langle Q_0[h]+\nabla Q_1[h], h \rangle \, \mathrm{d}x \\ \nonumber
&\leq \int_{\R^n} \left( \Delta \eta_R \right)|h|_{\delta}^p + \eta_R p |h|_{\delta}^{p-2} |h| |Q_0[h]| -  \langle Q_1[h], \nabla \left (\eta_R p |h|_{\delta}^{p-2} h \right) \rangle \, \mathrm{d}x \\ \nonumber
&\leq \int_{\R^n} \left( \Delta \eta_R \right)|h|_{\delta}^p + \eta_R p |h|_{\delta}^{p-2} |h| |Q_0[h]| + |Q_1[h]|\Big( p |\nabla \eta_R| |h|_{\delta}^{p-2} |h| \\ \nonumber
&\qquad\qquad \qquad + p \eta_R |h|_{\delta}^{p-2}|\nabla h| + \eta_R|p(p-2)||h|_{\delta}^{p-4}|h|^2|\nabla h| \Big) \, \mathrm{d}x. \nonumber
\end{align}
Recalling the estimates of $Q_0[h]$ and $Q_1[h]$ in (\ref{Q0-est}) and (\ref{Q1-est}) respectively, we see that
\begin{align}
\partial_t \int_{\R^n} \eta_R|h|_{\delta}^p  &\leq  \int_{\R^n} ( \Delta \eta_R )|h|_{\delta}^p + C \eta_R |h|_{\delta}^{p-2} |h| |\nabla h|^2 \\ \nonumber
& \qquad\qquad + C |\nabla \eta_R| |h|_{\delta}^{p-2} |h|^2|\nabla h| + C \eta_R  |h|_{\delta}^{p-4}|h|^3|\nabla h|^2  \, \mathrm{d}x\nonumber \\
&\leq \int_{\R^n} \left( \Delta \eta_R + 2\left(\nabla \sqrt{\eta_R}\right)^2 \right)|h|_{\delta}^p + C \eta_R |\nabla h|^2 \, \mathrm{d}x\\
& \leq \frac{C}{R^2}\int_{B_{2R}\setminus B_{R}} |h|^p_{\delta} \, \mathrm{d}x + C \int_{\R^n} \eta_R |\nabla h|^2\, \mathrm{d}x,
\end{align}
where we applied Young's inequality 
\begin{equation}
|\nabla \eta_R | |\nabla h| \leq \frac{1}{2} \eta_R |\nabla h|^2 + 2\left(\nabla \sqrt{\eta_R}\right)^2
\end{equation}
and used the assumption that $|h| < \frac{1}{2}$. Taking $\delta \rightarrow 0$ we obtain the desired result.

\end{proof}

We will now estimate $\norm{\nabla h}_{L^{2}(\R^n \times (0,\infty))}$ in terms of $\norm{h_0}_{L^2(\R^n)}$. By comparison to the standard heat equation $\partial_t u = \Delta u$ on $\R^n$ we expect such an estimate to exist: Multiplying the heat equation by $u$ we have
\begin{equation}
\frac{1}{2} \partial_t u^2 = u\Delta u
\end{equation}
Thus integrating by parts and rearranging we have
\begin{align}
\int_0^t \int_{\R^n} |\nabla u |^2 \,\mathrm{d}x \,\mathrm{d}t &= \frac{1}{2} \int_{\R^n} u(\cdot,0)^2 \,\mathrm{d}x- \frac{1}{2} \int_{\R^n} u(\cdot, t)^2 \,\mathrm{d}x\\
&\leq \frac{1}{2} \int_{\R^n} u(\cdot,0)^2 \,\mathrm{d}x.
\end{align}

For the Ricci DeTurck equation (\ref{RdT-simplified}) we can perform an analogous computation to deduce:
\begin{lem}
\label{local-int-L2-bound}
Let $\frac{1}{2} > \tilde{\epsilon} >0$. Then for a solution $h \in C^{2,1}(\R^n \times [0,\infty))$ to the Ricci DeTurck equation (\ref{RdT-simplified}) satisfying $\norm{h}_{L^{\infty}(\R^n \times [0,\infty))} <  \tilde{\epsilon}$ we have
\begin{equation}
\label{int-L2-bound}
\frac{1}{2} \partial_t \int_{\R^n} \eta_R |h|^2 \, \mathrm{d}x\leq (-\frac{1}{2} + C_3\tilde{\epsilon}) \int_{\R^n} \eta_R |\nabla h|^2 \, \mathrm{d}x+ \frac{C_3}{R^2}\int_{B_{2R}\setminus B_{R}} |h|^2 \, \mathrm{d}x,
\end{equation}
where $C_3$ depends on $n$ only.
\end{lem}

\begin{proof}
Multiplying the Ricci DeTurck equation (\ref{RdT-simplified}) by $\eta_R h$ and integrating by parts we obtain
\begin{align}
\frac{1}{2} \partial_t \int_{\R^n} \eta_R |h|^2 \, \mathrm{d}x &= \int_{\R^n} -\nabla \left(\eta_R h\right) \nabla h +  \eta_R \langle h, Q_0[h] \rangle - \langle \nabla (\eta_R h), Q_1[h]\rangle \big) \, \mathrm{d}x \\
&\leq \int_{\R^n} -\eta_R |\nabla h|^2 + |\nabla \eta_R| |h| |\nabla h| + \eta_R |h| |Q_0[h]| \\
& \qquad\quad\qquad\qquad + \eta_R |\nabla h| |Q_1[h]| + |\nabla \eta_R||h| |Q_1[h]| \, \mathrm{d}x.
\end{align}
By Young's inequality we have
\begin{equation}
|\nabla \eta_R | |h| |\nabla h| = 2 |\nabla \sqrt{\eta_R}| |h| \sqrt{\eta_R}|\nabla h| \leq 2|\nabla \sqrt{\eta_R}|^2 |h|^2 + \frac{1}{2} \eta_R |\nabla h|^2.
\end{equation}
Applying estimates (\ref{Q0-est}) and (\ref{Q1-est}) for $Q_0[h]$ and $Q_1[h]$ respectively, we thus obtain
\begin{align}
\frac{1}{2} \partial_t \int_{\R^n} \eta_R |h|^2 \, \mathrm{d}x & \leq  \int_{\R^n} -\frac{1}{2} \eta_R |\nabla h|^2 + 2 |\nabla \sqrt{\eta_R}|^2 |h|^2 \\
\nonumber &\qquad\qquad+ C \eta_R |h| |\nabla h|^2 + C |\nabla\eta_R||h|^2|\nabla h| \, \mathrm{d}x \\
&\leq \int_{\R^n} (-\frac{1}{2} + C\tilde{\epsilon} ) \eta_R |\nabla h|^2 + 2 |\nabla \sqrt{\eta_R}|^2 |h|^2 \\
\nonumber &\qquad\qquad + C\tilde{\epsilon} \left( 2|\nabla \sqrt{\eta_R}|^2 |h|^2 + \frac{1}{2} \eta_R |\nabla h|^2 \right) \, \mathrm{d}x \\
&\leq(-\frac{1}{2} + C\tilde{\epsilon} )\int_{\R^n} \eta_R |\nabla h|^2 \, \mathrm{d}x+ \frac{C}{R^2} \int_{B_{2R}\setminus B_R} |h|^2\, \mathrm{d}x
\end{align}
\end{proof}

Using the above estimate we can bound the $L^2$ norm of $|\nabla h|$ over space and time. 
\begin{cor}
\label{gradient-h-l2-bound}
 There exists an $\frac{1}{2}>\tilde{\epsilon}>0$ such that under the same conditions as in theorem \ref{main-lp-thm} for the case $1 \leq p\leq2$ we have 
\begin{equation}
\int_0^{\infty} \int_{\R^n} |\nabla h|^2 \, \mathrm{d}x \leq 2 \int_{\R^n} |h_0|^p\, \mathrm{d}x.
\end{equation}
\end{cor}

\begin{proof} We will prove the case $p=2$ from which the result follows by the observation that
\begin{equation}
\int_{\R^n} |h_0|^2 \, \mathrm{d}x \leq \int_{\R^n} |h_0|^p \, \mathrm{d}x < \infty
\end{equation}
for $\norm{h}_{L^{\infty}(\R^n\times[0,\infty))}\leq1$ and $ 1\leq p \leq 2$. By theorem \ref{main-lp-thm} case $p=2$, which we proved above, we can pick $\tilde{\epsilon} >0$ such that
\begin{equation}
\int_{\R^n} |h_t|^2 \, \mathrm{d}x \leq \int_{\R^n} |h_0|^2 \, \mathrm{d}x
\end{equation}
for all $t>0$. Integrating the estimate from lemma \ref{local-int-L2-bound} we have
\begin{align}
\frac{1}{2} \int_{\R^n} \eta_R |h|^2(x,t)\, \mathrm{d}x -  \frac{1}{2} \int_{\R^n} \eta_R |h_0|^2\, \mathrm{d}x &\leq (-\frac{1}{2}+ C_3 \tilde{\epsilon}) \int_0^t \int_{\R^n} \eta_R |\nabla h|^2 \, \mathrm{d}x\\ \nonumber
&\qquad\qquad + \frac{C_3}{R^2}\int_0^t\int_{B_{2R}\setminus B_{R}} |h|^2 \, \mathrm{d}x.
\end{align}
Taking $\tilde{\epsilon} < \frac{1}{4C_3}$ and $R\rightarrow \infty$, we obtain
\begin{equation}
\frac{1}{4} \int_0^t \int_{\R^n} |\nabla h|^2 \, \mathrm{d}x \leq \frac{1}{2} \int_{\R^n} |h_0|^2\, \mathrm{d}x.
\end{equation}
Since this inequality is independent of $t$ the desired result follows.
\end{proof}

Having control of $\nabla h$, we now proceed to proving theorem \ref{main-lp-thm} in the $1\leq p < 2$ case.

\begin{proof}[Proof of theorem (\ref{main-lp-thm}) for $1\leq p < 2$] Take $\tilde{\epsilon}>0$ small enough such that lemma \ref{local-lp-bound-p-less-2} and corollary \ref{gradient-h-l2-bound} hold. As previously, we define
\begin{equation}
A(t,R) = \sup_{x \in \R^n} \int_{\R^n} \eta_{R,x}(y) |h|^p(y,t) \, \mathrm{d}y.
\end{equation}
Integrating (\ref{lp-less-2-estimate}) and covering $B_{2R}(x)\setminus B_{R}(x)$ by $N$ balls of radius $R$ as before we obtain
\begin{equation}
A(t,R)- A(0,R) \leq \frac{NC_2}{R^2} \int_0^t A(s,R) \, \mathrm{d}s + C_2 \int_0^t \int_{\R^n}\eta_R |\nabla h|^2.
\end{equation}
Applying our $L^2(\R^n\times(0,\infty))$ bound of $|\nabla h|$ from corollary \ref{gradient-h-l2-bound} and rearranging we deduce
\begin{equation}
A(t,R) \leq (2C_2 +1) \norm{h_0}^p_{L^p(\R^n)} +\frac{NC_2}{R^2} \int_0^t A(s,R) \, \mathrm{d}s.
\end{equation}
By applying Grönwall's inequality (see appendix) and taking $R \rightarrow \infty$, we obtain the desired result.
\end{proof}

\subsection{$L^{\infty}$ decay rates} In the following we will utilize the $L^p$ bounds derived above to prove $L^{\infty}$ decay rates for solutions $h_t$ to the Ricci DeTurck flow (\ref{RdT-simplified}). The rates we obtain are consistent with what one would expect by comparison with the initial value problem for the standard linear heat equation $\partial_t u = \Delta u$ on $\R^n$:
Assuming we can represent the solution as 
\begin{equation}
u(x,t) = \int_{\R^n} K(x,t,y,0) u(y,0) \, \mathrm{d}y,
\end{equation}
where 
\begin{equation}
K(x,t,y,s) = (4\pi (t-s))^{-\frac{n}{2}} \exp\left(-\frac{|x-y|^2}{4(t-s)}\right)
\end{equation}
is the standard heat kernel on $\R^n$, we can use Young's inequality to estimate
\begin{align*}
\norm{h(\cdot, t)}_{L^{\infty}(\R^n)} &\leq \norm{K}_{L^{\frac{p}{p-1}}(\R^n)} \norm{u(\cdot, 0)}_{L^p(\R^n)} \\
&\leq \frac{C}{t^{\frac{n}{2p}}} \norm{u(\cdot, 0)}_{L^p(\R^n)}
\end{align*}
for some constant $C>0$. We will show that the same result holds for Ricci DeTurck flow:
\begin{lem}
\label{infty-decay-rates}
There exists constants $\epsilon'= \epsilon'(n)>0$ and $C'=C'(n)>0$ such that for a solution $h_t$ to the Ricci DeTurck equation (\ref{RdT-simplified}) satisfying (i) $\norm{h_0}_{L^{\infty}(\R^n)} < \epsilon'$ and (ii) $\norm{h_0}_{L^p(\R^n)} < \infty$ for some $1 \leq p < \infty$, we have 
\begin{equation}
\norm{h_t}_{L^{\infty}(\R^n)} <  \norm{h_0}_{L^p(\R^n)} \frac{C'}{t^{\frac{n}{2p}}}, \qquad t>0.
\end{equation}
\end{lem} 

Before we prove above lemma we will state a useful interpolation inequality.
\begin{lem}
\label{interpolation_inequality}
Let $f \in L^p(\R^n)$ for some $p \in [1,\infty)$ and $Df \in L^{\infty}(\R^n)$. Then there exists a constant $C(n,p)>0$ such that
\begin{equation}
\norm{f}_{L^{\infty}(\R^n)} \leq C(n,p) \norm{f}^{\frac{p}{n+p}}_{L^p(\R^n)} \norm{Df}^{\frac{n}{n+p}}_{L^{\infty}(\R^n)}.
\end{equation}
In particular, we can choose $C(n,p) = \left(\frac{(p+1) \cdots (p+n)}{\omega_n n!}\right)^{\frac{1}{n+p}}$.
\end{lem}

\begin{proof}
Let $x_0 \in \R^n$ and w.l.o.g. assume $f(x_0) > 0$. If we set $r_{x_0} = \frac{f(x_0)}{\norm{Df}_{L^{\infty}(\R^n)}}$, then for any $x$ such that $|x-x_0| \leq r_{x_0}$ we have
\begin{equation}
|f(x)| \geq \left| f(x_0) - \norm{Df}_{L^{\infty}(\R^n)} |x-x_0| \right|.
\end{equation}
Integrating we get
\begin{align}
\int_{\R^n} |f(x)|^p \, \mathrm{d}^nx &\geq \int_{|x-x_0| \leq r_{x_0}} \left| f(x_0) - \norm{Df}_{L^{\infty}(\R^n)} |x-x_0| \right|^p \, \mathrm{d}^nx \\ \nonumber
 &\geq n\omega_n \int_{0}^{r_{x_0}} r^{n-1} \left( f(x_0) - \norm{Df}_{L^{\infty}(\R^n)}r \right)^p \, \mathrm{d}r \\ \nonumber
 &\geq n\omega_n r_{x_0}^n f(x_0)^p \left ( \int_0^1 s^{n-1}(1-s)^p \, \mathrm{d}s \right),
\end{align}
where in the last line we used the substitution $r = r_{x_0} s$ and $\omega_n$ denotes the volume of the unit ball in $n$ dimensions. After evaluating the integral, re-substituting for $r_{x_0}$ and rearranging the last inequality, we obtain the desired result.
\end{proof}

We can now prove lemma \ref{infty-decay-rates}.

\begin{proof}[Proof of lemma \ref{infty-decay-rates}] Choose $\epsilon >0, C>0$ as in section \ref{section-KL} and $\tilde{\epsilon}>0, \tilde{C}>0$ as in theorem \ref{main-lp-thm}. Take $\epsilon' =\frac{1}{C} \min(\tilde{\epsilon}, \epsilon)$. Then by estimate (\ref{KL-decay-rates}) we have $\norm{h_t}_{L^{\infty}(\R^n)} < \min(\tilde{\epsilon}, \epsilon)$ for all $t\geq0$. Thus we can apply theorem \ref{main-lp-thm} in conjunction with lemma \ref{interpolation_inequality} above to deduce
\begin{equation}
\label{interpolation}
\norm{h_t}_{L^{\infty}(\R^n)} \leq C(n,p) \tilde{C}^{\frac{p}{n+p}} \norm{h_0}_{L^{p}(\R^n)}^{\frac{p}{n+p}}  \norm{D h_t}_{L^{\infty}(\R^n)}^{\frac{n}{n+p}},
\end{equation}
By the decay estimate (\ref{KL-decay-rates}) it follows that for $t>0$
\begin{equation}
\label{deriv_decay}
\sup_{x\in \R^n} |D h_t(x) | \leq \frac{CR}{t^{\frac{1}{2}}} \norm{h_0}_{L^{\infty}(\R^n)}.
\end{equation}
Combining the above inequalities we therefore obtain
\begin{equation}
\norm{h_t}_{L^{\infty}(\R^n)} \leq C(n,p) \tilde{C}^{\frac{p}{n+p}}(CR\epsilon')^{\frac{n}{n+p}} \norm{h_0}_{L^{p}(\R^n)}^{\frac{p}{n+p}} t^{-\frac{n}{2(n+p)}}.
\end{equation}
 As noted above $\norm{h_t}_{L^{\infty}(\R^n)}<\epsilon$ for all $t\geq0$, so we may start the flow from $\frac{t}{2}$ and reapply the estimate (\ref{KL-decay-rates})
\begin{equation}
\label{recursion}
|Dh_t(x)| \leq \frac{CR}{\left(\frac{t}{2}\right)^{\frac{1}{2}}} \norm{h_{\frac{t}{2}}}_{L^{\infty}(\R^n)}.
\end{equation}
Using (\ref{recursion}) and (\ref{interpolation}) recursively we can successively obtain better decay rates for $h$. Let $(C_i,\alpha_i, \beta_i)$ be such that after the $i$-th iteration we have 
\begin{equation}
|h_t(x)| \leq C_i t^{\alpha_i} \norm{h_0}_{L^{p}(\R^n)}^{\beta_i}.
\end{equation} 
Then 
\begin{equation}
(C_1,\alpha_1, \beta_1) = \left(C(n,p) \tilde{C}^{\frac{p}{n+p}}(CR\epsilon')^{\frac{n}{n+p}}, -\frac{n}{2(n+p)}, \frac{p}{n+p} \right)
\end{equation}
and at the $i$-th recursion step we obtain
\begin{align}
|Dh_t(x)| &\leq \frac{CR} { (\frac{t}{2})^{\frac{1}{2}} } \norm{h_{\frac{t}{2}}}_{L^{\infty}(\R^n)} \\
					\nonumber&\leq \frac{2^{\frac{1}{2}} RC}{t^{\frac{1}{2}}} C_i \left(\frac{t}{2}\right)^{\alpha_i}\norm{h_0}_{L^{p}(\R^n)}^{\beta_i} \\
					\nonumber&\leq RCC_i 2^{\frac{1}{2}-\alpha_i} t^{\alpha_i - \frac{1}{2}}\norm{h_0}_{L^{p}(\R^n)}^{\beta_i}
\end{align}
and therefore 
\begin{equation}
\norm{h_t}_{L^{\infty}(\R^n)} \leq C(n,p)\tilde{C}^{\frac{p}{n+p}} \left(2^{\frac{1}{2}-\alpha_i}RCC_i\right)^{\frac{n}{n+p}} t^{\frac{n}{n+p}(\alpha_i - \frac{1}{2})} \norm{h_0}_{L^{p}(\R^n)}^{\frac{p}{n+p} + \beta_i \frac{n}{n+p}}.
\end{equation}
From above we can read off the following recursion relations for $\alpha_i$, $\beta_i$ and $C_i$: 
\begin{align}
C_{i+1} &= C(n,p)\tilde{C}^{\frac{p}{n+p}} (2^{\frac{1}{2}-\alpha_i}RCC_i)^{\frac{n}{n+p}}  \\
\alpha_{i+1} &= \frac{n}{n+p}\left(\alpha_i - \frac{1}{2}\right) \\
\beta_{i+1} &= \frac{p}{n+p} + \beta_i \frac{n}{n+p}
\end{align}
As $i \rightarrow \infty$ we have
\begin{align}
\alpha_i &\rightarrow -\frac{n}{2p}\\
\beta_i &\rightarrow 1\\
C_i &\rightarrow C_{\ast} := 2^{\frac{n(n+p)}{2p^2}} \left(C(n,p)\right)^{\frac{n+p}{p}} (RC)^{\frac{n}{p}}.
\end{align} Using the expression for $C(n,p)$ we can bound $C_{\ast} \leq C'(n)$, for $C'(n)$ a constant depending on $n$ only. Thus the desired result holds true.
\end{proof}

We also get the following corollary of the above lemma.

\begin{cor}\label{decay-cor}
Under the same conditions as in lemma \ref{infty-decay-rates} we have for any $k\in \N_0$ and multi-index $\alpha \in \N^n_0$
\begin{equation}
\sup_{x\in \R^n}\sup_{t >0} t^{\frac{n}{2p}} | (t^{\frac{1}{2}}\nabla)^{\alpha}(t\partial_t)^k(g-g_{eucl})(x,t)| \leq \norm{h_0}_{L^p(\R^n)} C'(n) R^{|\alpha|+k}(|\alpha| + k)!,
\end{equation}
where $R>0$ is as in (\ref{KL-decay-rates}) and $C'(n)$ is a constant depending on $n$ only.
\end{cor}
\begin{proof}
Apply the derivative estimate (\ref{KL-decay-rates}) at time $\frac{t}{2}$ with the bound on $\norm{h_{\frac{t}{2}}}_{L^{\infty}(\R^n)}$ from lemma \ref{infty-decay-rates} above.
\end{proof}

Combining the results from \cite[Theorem 4.3]{KochAndLamm}, theorem \ref{regularity-RdT}, theorem \ref{main-lp-thm} and corollary \ref{decay-cor} above yields theorem \ref{main-thm}.

\section{Positive scalar curvature rigidity}
In this section we prove theorem \ref{scalar-rigidity-thm} by evolving the metric perturbation via Ricci DeTurck flow (\ref{RdT-simplified}). As outlined in the introduction, the proof strategy is to show that positive scalar curvature decays at a rate no faster than $O(t^{-\frac{n}{2}})$ during the flow. This is as one would expect, as the scalar curvature satisfies the super heat equation 
\begin{equation}
\label{R-RdT-eqn}
\partial_t R_{g_t} = \Delta_{g_t} R_{g_t} - \partial_{X(g_t)} R_{g_t} + 2|Ric(g_t)|_{g_t}^2
\end{equation}
and the $L^{\infty}$ decay rate of the standard heat kernel on $\R^n$ is of order $O(t^{-\frac{n}{2}})$. However, the decay rates obtained in corollary \ref{decay-cor} and the coordinate expression (\ref{scalar-curvature-coord}) for $R$ show that the scalar curvature decays at the rate $O(t^{-\frac{n}{2p}-1})$. This proves that no positive scalar curvature can be present during the evolution of metric perturbations bounded in $L^p$ for some $p<\frac{n}{n-2}$. We will show that this implies that the metric perturbation is flat in this case. 

The new ingredient in this section is Perelman's Harnack inequality, which we use to obtain a lower bound for the heat kernel on a Ricci flow background. Because in the non-compact setting Perelman's Harnack inequality requires uniform curvature bounds, which we do not have at $t=0$, we first run Ricci DeTurck flow for a short time $t_0>0$ to obtain uniform curvature bounds. Then we evolve the metric via the Ricci flow from  time $t_0$ onwards and apply Perelman's Harnack inequality. The only technical difficulty is to show that the scalar curvature remains non-negative at early times, because on non-compact manifolds without curvature bounds we cannot directly apply the maximum principle to the super heat equation (\ref{R-RdT-eqn}).

For the rest of this section we will fix an $\epsilon>0$ and $C>0$ such that theorem \ref{main-thm} holds. We will also assume that we are given a smooth metric $g$ that satisfies 
\begin{enumerate}
\item $R_g \geq 0$
\item $\norm{g-g_{eucl}}_{L^{\infty}(\R^n)} < \epsilon$ 
\item $\norm{g-g_{eucl}}_{L^p(\R^n)} < \infty$ for some $p \in [1,\infty)$
\end{enumerate}

\subsection{ Conservation of non-negative scalar curvature} 
Evolving the metric from initial data $g_{0}=g$ via Ricci DeTurck flow, we obtain a family of metrics $(g_t)_{t\geq0}$ solving (\ref{RdT}). By theorem $\ref{RdT-smooth}$ this solution is smooth on $\R^n \times [0,\infty)$ and all derivatives of the solution converge locally to the initial data as $t\rightarrow 0$. In the lemma below we prove that for initial data with $R(g_0)\geq 0$ the scalar curvature $R(g_t)$ remains non-negative for all times $t\geq0$.
\begin{lem}
Let $g$ be as above. Then the scalar curvature remains non-negative in the Ricci DeTurck flow starting from $g$.
\end{lem}
\begin{proof}
We can apply \cite[Lemma 4]{RB2}, because the assumption that $g-g_{eucl}$ is compactly supported was only used to prove that the first and second derivatives of $g_t$ vary continuously as $t\rightarrow0$. By parabolic dilations, \cite[Lemma 4]{RB2} implies that for every $\delta>0$ there is a $\tau = \tau(\delta)$ such that the following is true: Assume that $R(g_{0}) > a$ on $B(o,r)$ for some $a\in\R$. Then the solution $(g_t)_{t\geq0}$ to the Ricci DeTurck flow equation (\ref{RdT-simplified}) satisfies $R(o,t) > a - \frac{\delta}{r^2}$ for all $t \in [0, \tau r^2]$. Since we can take any $a<0$ and $r>0$ the result follows. 
\end{proof}

\begin{remark}
We did not use the condition $\norm{g-g_{eucl}}_{L^p(\R^n)} < \infty$ in the proof of the lemma above.
\end{remark}
\begin{remark}
One can show using \cite[Lemma 4]{RB2} that there are constants $\epsilon>0$ and $C>0$ such that no metric $g$ on $B_1(0) \subset \R^n$ exists that satisfies $\norm{g-g_{eucl}}_{L^{\infty}(B_1(0))} < \epsilon$ and $R_g >C$ everywhere on $B_1(0)$.
\end{remark}

\subsection{Lower heat kernel bounds on Ricci flow background} Fix $t_0>0$ and let 
\begin{equation}
\tilde{g}_{t_0 + t} = (\Phi_{t_0+t})^{\ast} g_{t_0+t}, \quad t\geq0,
\end{equation}
be the corresponding Ricci flow starting from $g_{t_0}$ (see section \ref{RicciFlow}). Because $(\tilde{g}_{t_0+t})_{t\geq0}$ and $(g_{t_0 + t})_{t\geq0}$ are related by diffeomorphisms, we obtain by corollary \ref{decay-cor} the uniform curvature bounds 
\begin{equation}
\label{curv-decay}
|R_{\tilde{g}_t}|_{\tilde{g}_t}, \, |Rc_{\tilde{g}_t}|_{\tilde{g}_t},\, |Rm_{\tilde{g}_t}|_{\tilde{g}_t} < \frac{C_2 \norm{g-g_{eucl}}_{L^p(\R^n)} }{(t_0+t)^{1+\lambda}},
\end{equation}
where $\lambda = \frac{n}{2p}$ and $C_2$ is a constant depending on $C$ and $n$ only. By the Ricci flow equation we have
\begin{equation}
\partial_t \tilde{g}_t(\xi,\xi) = -2 Ric_{\tilde{g}_t}(\xi,\xi)
\end{equation}
for any fixed $\xi \in T\R^n$. Therefore we obtain
\begin{equation}
|\partial_t \tilde{g}_t(\xi,\xi)| \leq 2 |Ric_{\tilde{g}_t}|_{\tilde{g}_t} \tilde{g}_t(\xi,\xi) \leq \frac{2 C_2 \norm{g-g_{eucl}}_{L^p(\R^n)} }{(t_0+t)^{1+\lambda}}\tilde{g}_t(\xi,\xi).
\end{equation}
Integrating this bound allows us to estimate $\tilde{g}_{t_0 + t}$ in terms of $\tilde{g}_{t_0}$
\begin{equation}
\label{bilip}
e^{-C_3} \tilde{g}_{t_0} \leq \tilde{g}_{t_0 +t} \leq e^{C_3} \tilde{g}_{t_0}, \quad t\geq0,
\end{equation} 
where 
\begin{equation}
C_3 = \frac{2C_2 \norm{g-g_{eucl}}_{L^p(\R^n)}}{\lambda t_0^{\lambda}}.
\end{equation}
Furthermore note that from the bound $\norm{h_t}_{L^{\infty}(\R^n)} \leq C \norm{h_0}_{L^{\infty}(\R^n)}$ one obtains by Cauchy-Schwarz
\begin{equation}
\norm{\xi}_{\tilde{g}_{t_0}}^2 \leq (1+C n \epsilon) \norm{\xi}_{g_{eucl}}^2.
\end{equation}

On the Ricci flow background the scalar curvature satisfies a super heat equation related to (\ref{R-RdT-eqn})
\begin{equation}
\label{R-Ricci-eqn}
\partial_t R_{\tilde{g}_{t_0+t}} = \Delta_{\tilde{g}_{t_0+t}} R_{\tilde{g}_{t_0+t}} + 2 |Ric(\tilde{g}_{t_0+t})|^2_{\tilde{g}_{t_0+t}}, \quad t\geq 0.
\end{equation}
In order to analyse the decay rate of $R_{\tilde{g}_{t_0+t}}$ we need a lower bound on the heat kernel $K_{\tilde{g}}(x,t;y,s)$ of the linear heat equation 
\begin{equation}
\label{evolving-heat-eqn}
\partial_t u = \Delta_{\tilde{g}_{t_0+t}} u, \qquad t\geq 0,
\end{equation}
on the Ricci flow background $(\tilde{g}_{t_0+t})_{t\geq0}$. For this we rely on Perelman's Harnack inequality (see \cite[Section 9]{Perelmann}), the basic setup of which we will recap here: Let $x,y\in \R^n$ and $0<s<t<T$, then the $\Lcurv$-length of a curve $\gamma: [s,t] \rightarrow \R^n$ is defined as
\begin{equation}
\Lcurv(\gamma) := \int_s^t \sqrt{t-t'}\Big(R(\gamma(t'),t') + |\dot{\gamma}(t')|^2_{t'}\Big) \, \mathrm{d}t'
\end{equation}
and the reduced distance between $(x,t)$ and $(y,s)$ is defined as 
\begin{equation}
l_{(x,t)}(y,s) := \frac{1}{2\sqrt{t-s}} \inf \{\Lcurv(\gamma) |\; \gamma: [s,t] \rightarrow \R^n \; \text{between $(x,t)$ and $(y,s)$} \}.
\end{equation}
On a closed manifold we can use the reduced distance to find a lower bound of the heat kernel (see \cite[Lemma 16.49]{ChowII}). Below we show that this carries over to our non-compact setting.
\begin{lem} In the same notation as above we have for $t>s\geq0$
\begin{equation}
	\label{RF-heat-kernel}
	K_{\tilde{g}}(x,t,y,s) \geq \frac{1}{(4\pi(t-s))^{\frac{n}{2}}}\exp\left(-l_{(x,t)}(y,s)\right).
\end{equation}
\end{lem}
\begin{proof}
 Because of the uniform curvature bounds (\ref{curv-decay}) for $(\tilde{g}_{t_0+t})_{t\geq0}$ on $\R^n \times [0, \infty)$, Perelman's comparison geometry for $\Lcurv$ still holds true \cite[pp.322]{ChowI}. The curvature bounds in conjunction with the bilipshitz bounds (\ref{bilip}) ensure that the weak maximum principle for the heat equation on $(\R^n, (\tilde{g}_{t_0+t})_{t\geq0})$ holds true by \cite[pp. 140]{ChowII}. Therefore the proofs of \cite[Lemma 16.48]{ChowII}  and \cite[Lemma 16.49]{ChowII} carry over to our non-compact setting.
\end{proof}
The reduced distance $l_{(x,t)}(y,s)$ can be bounded by
\begin{align}
l_{(x,t)}(y,s) &\leq \frac{1}{2\sqrt{t-s}} \int_s^t \sqrt{t-t'} \left(\frac{C_2 \norm{g-g_{eucl}}_{L^p(\R^n)} }{(t_0+t')^{1+\lambda}} + e^{C_3}(1+C\epsilon n)|\dot{\gamma}(t')|^2_{g_{eucl}} \right) \, \mathrm{d}t'  \\ \nonumber
\label{integral}				
&\leq\norm{g-g_{eucl}}_{L^p(\R^n)}  \frac{C_2}{2}\int_s^t  \frac{1}{(t_0+t')^{1+\lambda}}\, \mathrm{d}t' + \frac{e^{C_3}}{3}(1+C\epsilon n) \frac{|x-y|^2}{t-s}\\ \nonumber
&\leq  \frac{C_2}{2\lambda} \frac{\norm{g-g_{eucl}}_{L^p(\R^n)}}{(t_0+s)^{\lambda}} + \frac{e^{C_3}}{3}(1+C\epsilon n) \frac{|x-y|^2}{t-s},
\end{align}
where we took $\gamma$ to be the Euclidean geodesic with constant velocity connecting $x$ and $y$.

We may summarize the above in the following lemma:
\begin{lem}
\label{RF-heat-kernel-bound} In the notation from above, let $(g_t)_{t\geq 0}$ be a solution to the Ricci DeTurck equation (\ref{RdT}) starting from $g_0 = g$ and let $(\tilde{g}_{t_0+t})_{t\geq0}$ be the corresponding Ricci flow starting from $g_{t_0}$. Then for $t > s \geq 0 $ we have
 \begin{equation}
 \label{HK}
K_{\tilde{g}}(x,t,y,s) \geq \frac{C_4}{\left(4\pi(t-s)\right)^{\frac{n}{2}}}\exp\left(- C_5 \frac{|x-y|^2}{t-s}\right),
 \end{equation}
 where 
\begin{equation}
C_4 = \exp\left(-\frac{C_3}{4}\right)
\end{equation}
and
\begin{equation}
C_5 = \frac{e^{C_3}}{3}(1+C\epsilon n).
\end{equation}
\end{lem}
\begin{proof}
See above.
\end{proof}

\subsection{Proof of positive scalar curvature rigidity}
Now we proceed to prove the main theorem of this section.

\begin{proof}[Proof of theorem \ref{scalar-rigidity-thm}]
Let $(g_t)_{t\geq 0}$ be the Ricci DeTurck flow starting from $g_0 = g$ as above. Assume that at some point $(x_0,t_0)\in R^n \times (0,\infty)$ we have $R_0:=R(g_{t_0})(x_0) > 0$. By translation we may assume without loss of generality that $x_0 = o$. Then consider the Ricci flow $(\tilde{g}_{t_0+t})_{t\geq0}$ starting from the initial metric $g_{t_0}$. Pick $\delta>0$ such that $R(\tilde{g}_{t_0})\geq \frac{R_0}{2}$ on $B_o(\delta)$. Then from (\ref{R-Ricci-eqn}) and lemma \ref{RF-heat-kernel-bound} we have for $t \geq 1$
\begin{align}
R(\tilde{g}_{t_0+t})(o) &\geq \int_{\R^n} K_{\tilde{g}}(0,t,y,0) R(\tilde{g}_{t_0})(y) \, \mathrm{d}V_{\tilde{g}_{t_0}} \\ \nonumber
				  &\geq \frac{R_0}{2} \int_{B_o(\delta)} K_{\tilde{g}}(0,t,y,0) \, \mathrm{d}V_{\tilde{g}_{t_0}} \\ \nonumber
				  &\geq \frac{R_0}{2} \int_{B_o(\delta)}  \frac{C_4}{\left(4\pi t\right)^{\frac{n}{2}}}\exp\left(-C_5\frac{|y|^2}{t}\right) \, \mathrm{d}V_{\tilde{g}_{t_0}} \\ \nonumber
				  &\geq \frac{C_6}{t^{\frac{n}{2}}},
\end{align}
where $C_6>0$ is a constant independent of $t$. For $1 \leq p < \frac{n}{n-2}$ this contradicts the curvature decay rates (\ref{curv-decay}) and we deduce that $R(g_t) = 0$ for all $t>0$. Therefore $|Ric(g_t)| = 0$ for $t\geq0$ by equation (\ref{R-Ricci-eqn}). Because Ricci DeTurck flow is related to Ricci flow by a family of diffeomorphisms and Ricci flow is stationary for $Ric = 0$, we deduce by the curvature decay rates (\ref{curv-decay}) that our initial metric $g$ must have been flat. 
\end{proof}

In the borderline case case $p = \frac{n}{n-2}$, $n\geq3$ we can prove that the scalar curvature $R(g_t)$ becomes bounded in $L^1$ for any $t>0$:

\begin{lem}
\label{borderline}
Let $g$ be a smooth metric satisfying the conditions of theorem \ref{scalar-rigidity-thm} in the case $p = \frac{n}{n-2}$. Let $(g_t)_{t\geq0}$ be a solution to the Ricci DeTurck flow (\ref{RdT-simplified}) starting from the initial data $g_0 = g$. Then for $t_0>0$
\begin{equation}
\int_{\R^n} R(g_{t_0})(x) \, \mathrm{d}x < C_7 \norm{g-g_{eucl}}_{L^p(\R^n)} \exp\left( C_7 \frac{\norm{g-g_{eucl}}_{L^p(\R^n)}}{t_0^{\frac{n}{2}-1}}\right),
\end{equation}
where $C_7>0$ is a constant depending on $C$ and $n$ only.
\end{lem}

\begin{proof}
As before we have
\begin{equation}
R(g_{t_0+t})(o) \geq \int_{\R^n} \frac{C_4}{(4\pi t)^{\frac{n}{2}}} \exp\left(-C_5 \frac{|x|^2}{t}\right) R(g_{t_0})(x) \, \mathrm{d}V_{g_{t_0}}.
\end{equation}
Recalling our curvature estimate 
\begin{equation}
R(g_{t_0+t})(o) < \frac{C_2\norm{g-g_{eucl}}_{L^p(\R^n)}}{(t_0 + t)^{1 + \lambda}}
\end{equation} 
gives
\begin{equation}
\label{ineq}
\frac{C_2}{C_4} (4\pi)^{\frac{n}{2}} \norm{g-g_{eucl}}_{L^p(\R^n)} > \int_{\R^n}   \exp\left(-C_5 \frac{|x|^2}{t}\right) R(g_{t_0})(x) \, \mathrm{d}V_{g_{t_0}},
\end{equation}
since $1 + \lambda = 1 + \frac{n}{2p} = \frac{n}{2}$ in the borderline case $p = \frac{n}{n-2}$. 
By (\ref{KL-decay-rates}) we have
\begin{equation}
\norm{g_{t_0}- g_{eucl}}_{L^{\infty}(\R^n)} < c \epsilon,
\end{equation} 
so for $\epsilon>0$ suitably small the volume elements $\mathrm{d}V_{g_{eucl}}$ and $\mathrm{d}V_{g_{t_0}}$ are comparable. The desired result follows by taking the limit $t\rightarrow \infty$ in (\ref{ineq}) and applying the monotone convergence theorem.
\end{proof}

\section*{Appendix}
\begin{gronwall}
Let $u$ and $f$ be continuous and non-negative functions defined on $I=[\alpha, \beta]$, and let $n(t)$ be a continous, positive, nondecreasing function defined on $I$; then
\begin{equation}
u(t)\leq n(t) + \int_{\alpha}^t f(s)u(s) \, \mathrm{d}s , \quad  t \in I,
\end{equation}
implies that 
\begin{equation}
u(t)\leq n(t)\exp\left(\int_{\alpha}^t f(s) \, \mathrm{d}s\right) , \quad t \in I.
\end{equation}
\end{gronwall}
\begin{proof}
See \cite{Gronwall}, theorem 1.3.1.
\end{proof}

\section*{Acknowledgments}
The author would like to thank his advisor Prof. Richard Bamler for suggesting this project and giving very helpful advice along the way. We also thank the referee for helpful comments.


\begin{thebibliography}{9}
\bibitem[ADM62]{ADM}
R. Arnowitt, S. Deser, C. W. Misner, \emph{The Dynamics of General Relativity},
Gravitation: An Introduction to Current Research, pages 227-265.
Wiley, New York, 1962.

\bibitem[Bam14]{RB1}
R. Bamler, \emph{Stability of Hyperbolic Manifolds with Cusps under Ricci Flow}, Advances in Mathematics
Volume 263, 1 October 2014, Pages 412–467

\bibitem[Bam16]{RB2}
R. Bamler, \emph{A Ricci Flow Proof of a Result by Gromov on
Lower Bounds for Scalar Curvature}, Math. Res. Letters, 23(2) (2016):325-337


\bibitem[Bar86]{Bartnik}
R. Bartnik, \emph{The Mass of an Asymptotically Flat Manifold}, Comm. Pure
Appl. Math.,39(5):661-693, 1986.

\bibitem[ChI]{ChowI}
B. Chow et al.,
\emph{The Ricci Flow: Techniques and Applications: Part I: Geometric Aspects},
Mathematical Surveys and Monographs, vol. 135. American Mathematical Society, Providence, RI (2007)

\bibitem[ChII]{ChowII}
B. Chow et al.,
\emph{The Ricci Flow: Techniques and Applications: Part II: Analytic Aspects},
Mathematical Surveys and Monographs, vol. 144. American Mathematical Society, Providence, RI (2008)


\bibitem[KL12]{KochAndLamm}
H. Koch und T. Lamm,
\emph{Geometric flows with rough initial data},
Asian J. Math. 16, no. 2, 209-235 (2012)

\bibitem[KYL96]{KYL}
N. V Krylov, \emph{Lectures on elliptic and parabolic equations in Hölder spaces}, vol. 12,
Graduate Studies in Mathematics (Providence, RI: American Mathematical Society,
1996).

\bibitem[Pach98]{Gronwall}
B.G. Pachpatte, 
\emph{Inequalities for differential and integral equations},
 San Diego: Academic Press (1998), ISBN 9780080534640.

\bibitem[Per02]{Perelmann}
G. Perelman, \emph{The entropy formula for the Ricci flow and its geometric applications},
http://arxiv.org/abs/math/0211159 (2002).

\bibitem[Shi89]{Shi89} W.-X. Shi, \emph{Deforming the metric on complete Riemannian manifolds}, J. Differential Geometry 30(1989) 223-301

\bibitem[Sma93]{Sma93} N. Smale, \emph{Generic regularity of homologically area minimizing hypersurfaces
in eight-dimensional manifolds}, Comm. Anal. Geom. 1 (1993), no. 2, 217–228.
MR 1243523

\bibitem[SY79a]{SY}
R. Schoen; S.T. Yau \emph{On the proof of the positive mass conjecture in general relativity}, Comm. Math. Phys. 65 (1979), no. 1, 45--76. 

\bibitem[SY79b]{SY2}
R. Schoen and S.T. Yau, \emph{Complete manifolds with nonnegative
scalar curvature and the positive action conjecture in general relativity}, Proc.
Nat. Acad. Sci. U.S.A. 76 (1979), no. 3, 1024–1025. MR 524327 (80k:58034)

\bibitem[SY17]{SY17} R. Schoen and S.-T. Yau, \emph{Positive Scalar Curvature and Minimal Hypersurface Singularities}, ArXiv e-prints (2017)

\bibitem[SSS08]{MilesSimon}
O. Schnürer, F. Schulze and M. Simon,
  \emph{Stability of euclidean space under Ricci flow},
  Comm. Anal. Geom. 16 (2008), no. 1, 127-158


\bibitem[Top06]{Topping}
 P. Topping,
 \emph{Lectures on the Ricci Flow},
 London Mathematical Society, Lecture Note series 325, Cambridge University Press (2006)


\end{thebibliography}
\end{document}